\documentclass[12pt,reqno]{amsart}
\usepackage[headings]{fullpage}
\usepackage{amssymb,amsmath,amscd,bbm,tikz-cd}
\usepackage[all,cmtip]{xy}
\usepackage{url}
\usepackage[bookmarks=true,%
colorlinks=true,%
linkcolor=blue,
citecolor=blue,%
filecolor=blue,%
menucolor=blue,%
urlcolor=blue,%
breaklinks=true]{hyperref}
\usepackage{slashed}
\usepackage{listings}
\usepackage{verbatim}
\usepackage{mathtools}
\usepackage[normalem]{ulem}   
\usepackage{graphicx}
\usepackage{texdraw}
\usepackage{caption}
\usepackage{multirow}
\usepackage{tabularx}
\usepackage{array}
\usepackage{longtable}
\usepackage{pgfplots}
\pgfplotsset{compat=1.17}
\graphicspath{{figures/}}

\newcommand{\BA}{\mathbb{A}}

\newcommand{\BC}{\mathbb{C}}

\newcommand{\calN}{\mathcal{N}}
\newcommand{\calQ}{\mathcal{Q}}
\newcommand{\calX}{\mathcal{X}}

\newcommand{\calG}{\mathcal{G}}
\newcommand{\calR}{\mathcal{R}}

\newcommand{\sgn}{\mathrm{sgn}}
\newcommand{\SL}{\mathrm{SL}}
\newcommand{\PSL}{\mathrm{PSL}}

\newcommand{\Tr}{\mathrm{Tr}}

\newcommand{\Spec}{\mathrm{Spec}}

\newcommand{\lrbar}[1]{\langle #1 \rangle}

\newcommand{\II}{\mathrm{i}}   

\newtheorem{theorem}{Theorem}[section]
\newtheorem{corollary}[theorem]{Corollary}

\theoremstyle{definition}
\newtheorem{lemma}[theorem]{Lemma}
\newtheorem{definition}[theorem]{Definition}
\newtheorem{proposition}[theorem]{Proposition}

\newtheorem{remark}[theorem]{Remark}

\begin{document}
	
	\title[ ]{Computations of parabolic character schemes of knots}

    \author{Yunhi Cho}
	\address{Department of Mathematics, University of Seoul, Seoul, 02504, Korea}
	\email{yhcho@uos.ac.kr}
	
	\author{Hyuk Kim}
	\address{Department of Mathematical Sciences, Seoul National University, Seoul, 08826, Korea}
	\email{hyukkim@snu.ac.kr}

	\author{Seonhwa Kim}
	\address{Department of Mathematics, Natural Science Research Institute, University of Seoul, Seoul, 02504, Korea}
	\email{seonhwa17kim@gmail.com}
		
	\author{Seokbeom Yoon}
	\address{Department of Mathematics, Center for Quantum Computing, Chonnam National University, Gwangju, 61186, Korea}
	\email{sbyoon15@gmail.com}
		
	\keywords{Parabolic representations, parabolic quandles, parabolic character scheme, knot determinant}
	
	\date{\today}
	
	\begin{abstract}
		We compute parabolic $\mathrm{SL}_2(\mathbb{C})$-character schemes of knots using the parabolic quandle. To this end, we introduce sign-refined arc-colorings and show that their sign data encode the obstruction classes of the induced parabolic representations. We also establish a correspondence between the schemes defined by sign-refined arc-colorings and the parabolic character scheme. This yields a practical diagrammatic method for computing complete lists of parabolic characters, together with their multiplicities and obstruction classes. Using this method, we verify a conjecture of B\'{e}nard and Detcherry for all small knots with at most $12$ crossings.
	\end{abstract}
	\maketitle

    \section{Introduction}\label{sec:introduction}
	
	Parabolic representations of two-bridge knots were extensively studied by Riley \cite{riley_parabolic_1972,riley_parabolic_1975}. His explicit computations produced many important examples and played a significant role in the early development of three-dimensional geometry and topology, alongside Thurston's work on hyperbolic structures.
	
	One of the main difficulties in computing parabolic representations is that the defining equations quickly become complicated when written directly from the knot group. Riley's method is remarkably effective for two-bridge knots, but for general knots one needs a more flexible system of equations that is both practical and complete, in the sense that it detects all parabolic representations with a manageable amount of computation. The parabolic quandle provides such a diagrammatic framework for parabolic $\mathrm{PSL}_2(\BC)$-representations by encoding them as vector-colorings of a knot diagram; see, e.g., \cite{inoue_quandle_2013, cho_quandle_2018,JoKim}. 
	
	\subsection{Overview}
    In this paper, we employ the quandle method for computing the parabolic $\mathrm{SL}_2(\BC)$-character scheme of a knot. A direct adaptation of this framework leads to a sign ambiguity, since a vector and its negative determine the same matrix. We remove this ambiguity by assigning a sign to each crossing, leading to the notion of a sign-refined arc-coloring. Although this refinement is elementary, it admits a natural theoretical interpretation: the product of the signs assigned to the crossings agrees with the obstruction class of the associated parabolic representation (Theorem~\ref{thm:obs}).

    For each sign assignment, the defining equations of sign-refined arc-colorings determine a scheme. We show that it is a $\mu_2$-torsor over the corresponding obstruction-class component of the parabolic representation scheme (Theorem~\ref{prop.mu2torsor}). This provides a natural link between sign-refined arc-colorings and the parabolic representation scheme, and hence the parabolic character scheme. We then introduce an open cover of the parabolic character scheme such that, on each open chart, every $\SL_2$-orbit of arc-colorings admits a unique normalized representative (Theorem~\ref{thm:normalizedchart}). This makes explicit computations of the parabolic character scheme feasible.
	
	We illustrate the method through examples, beginning with the standard example $4_1$ and then turning to more complicated cases such as $8_{18}$. In particular, we show that, with a manageable amount of computation, the knot $8_{18}$ has exactly $26$ parabolic characters: two with positive obstruction class and twenty-four with negative obstruction class. This agrees with a conjecture of  B\'{e}nard and Detcherry \cite{BD} for small knots:
	\begin{equation} \label{conj.bd}
		n_- - n_+ = \frac{\det(K) -1}{2}  \,.
	\end{equation}
	Here, $n_\pm$ respectively denote the numbers of nontrivial parabolic characters with positive and negative obstruction classes, counted with multiplicity\footnote{The conjecture fails if the multiplicity is ignored. Thus, verifying the conjecture requires computing the parabolic character scheme rather than the underlying character variety.}, and $\det(K)$ denotes the determinant of the knot $K$. Note that the smallness of $K$ ensures that $n_\pm$ are finite.

	\subsection{Experiment}
	With computer assistance, the third author and Philip Choi computed complete lists of parabolic characters for all $2977$ knots with at most $12$ crossings, together with their multiplicities and obstruction classes \cite{diagramsite}. The computations depend crucially on the order in which the variables are eliminated. After testing a large number of possible orderings by computer, we found an effective one and completed the computations for all such knots. These data yield the following observations:
	
	\begin{itemize}
		\item There are $2977$ knots with at most $12$ crossings. Among them, $93$ knots have positive-dimensional components in their parabolic character schemes. By \cite{culler_varieties_1983}, such knots are necessarily large. Further discussion of these knots can be found in \cite{CPY26}.
			
		\item There are $1958$ small knots with at most $12$ crossings, and we verified Equation~\eqref{conj.bd} for all of them. Interestingly, although the conjecture is formulated for small knots, it also holds for all but ten of the $926$ large knots whose parabolic character schemes are finite.
		
		\item The proportion of knots with at least one parabolic character of positive obstruction class increases rapidly with the crossing number. More precisely, among knots with $8$, $9$, $10$, $11$, and $12$ crossings, the numbers of such knots are $7$ out of $21$ ($33.3\%$), $23$ out of $49$ ($46.9\%$), $111$ out of $165$ ($67.3\%$), $448$ out of $552$ ($81.2\%$), and $1964$ out of $2176$ ($90.3\%$), respectively. In total, $2553$ of the $2977$ knots have at least one parabolic character with positive obstruction class.
		
		\item Every non-alternating knot with at most $12$ crossings admits a parabolic character with positive obstruction class; equivalently, it admits a boundary-unipotent $\mathrm{SL}_2(\BC)$-representation.
	\end{itemize}
	
    Another novelty of our method is its compatibility with the octahedral decomposition of knot complements, which has been extensively studied in \cite{kim_octahedral_2018,kim_octahedral_2019,mcphail-snyder_hyperbolic_2022,Mc25}. In particular, the resulting arc-colorings can be used directly in known diagrammatic formulas for cusp shapes and complex volumes. Numerical values of the cusp shape and complex volume for each parabolic character are available in \cite{diagramsite}.

		
	\subsection{Organization}
	The paper is organized as follows. In Section~\ref{sec.prelim}, we collect basic definitions and properties of parabolic characters and the parabolic quandle. In Section~\ref{sec.quandle}, we introduce a sign-refined version of the parabolic quandle, prove its relation with obstruction classes, and establish a correspondence between the schemes defined by sign-refined arc-colorings and the parabolic character schemes. In Section~\ref{sec.example}, we present explicit computations and illustrate how the quandle method can be used in practice. In Section~\ref{sec.bd}, we examine the B\'{e}nard--Detcherry conjecture for all knots with at most $12$ crossings.
    
    \section{Parabolic representations and the parabolic quandle} \label{sec.prelim}
	
	In this section, we collect some basic definitions and properties of parabolic representations and the parabolic quandle. For completeness, we present the necessary material in a self-contained manner; the reader may find \cite{heusener2023scheme,milne1980etale} useful.

	\subsection{Parabolic representations}
	
	Let $K$ be a knot in $S^3$, and  $\calG(K)$ denote its knot group $\pi_1(S^3\setminus K)$. We say that an $\SL_2(\BC)$-representation of $\calG(K)$ is \emph{parabolic} if every meridian is mapped to a matrix of trace $2$. Since all meridians are conjugate in $\calG(K)$, it suffices to require that a single meridian be mapped to a trace-$2$ matrix.
		
	Fix a diagram $D$ of $K$. The associated Wirtinger presentation of $\calG(K)$ is of the form
	$$
	\calG(K)=\langle g_1,\ldots,g_n\mid r_1,\ldots,r_n\rangle,
	$$
	where $n$ is the number of crossings in $D$, and the generators $g_i$ and relations $r_i$ arise from the arcs and crossings of $D$, respectively. Then, computing parabolic representations amounts to finding $n$ trace-$2$ matrices satisfying the Wirtinger relations. Precisely, for each $i=1,\ldots, n$, assign a copy $A_i$ of the ring
	\begin{align*}
		A=\BC[a,b,c,d]/(a d-b c-1, \, a+d-2)
	\end{align*}    
	to the $i$-th arc of $D$, and let $I\subset A^{\otimes n} = \bigotimes_{i=1}^n A_i$ be the ideal generated by the $4n$ equations arising from the crossings of $D$, with each crossing contributing four equations obtained by equating the matrix entries in the Wirtinger relation:
	\begin{equation} \label{eqn.wirtinger}
		\begin{matrix}
			\vcenter{\hbox{\begin{tikzpicture}
						\draw[-stealth,thick] (0,0) -- (1,1);
						\draw[line width=5pt,white]  (0,1)--(1,0);
						\draw[-stealth,thick]  (0,1)--(1,0) ;
						\node at (1,1) [xshift=1.5ex] {$j$};
						\node at (0,0) [xshift=-1ex] {$i$};
						\node at (1,0)  [xshift=1.5ex] {$k$};
					\end{tikzpicture}
			}}
			\longrightarrow
			\begin{pmatrix} a_j & b_j \\ c_j & d_j \end{pmatrix} =
			\begin{pmatrix} d_k & -b_k \\ -c_k & a_k \end{pmatrix}
			\begin{pmatrix} a_i & b_i \\ c_i & d_i \end{pmatrix} \begin{pmatrix} a_k & b_k \\ c_k & d_k \end{pmatrix}  , \\		
			\vcenter{\hbox{\begin{tikzpicture}
						\draw[-stealth,thick]  (0,1)--(1,0) ;
						\draw[line width=5pt,white]  (0,0)--(1,1);
						\draw[-stealth,thick] (0,0) -- (1,1);
						\node at (1,1) [xshift=1.5ex] {$k$};
						\node at (0,1) [xshift=-1ex] {$i$};
						\node at (1,0)  [xshift=1.5ex] {$j$};
					\end{tikzpicture}
			}}	\longrightarrow
			\begin{pmatrix} a_j & b_j \\ c_j & d_j \end{pmatrix} =
			\begin{pmatrix} a_k & b_k \\ c_k & d_k \end{pmatrix}
			\begin{pmatrix} a_i & b_i \\ c_i & d_i \end{pmatrix} 
			\begin{pmatrix} d_k & -b_k \\ -c_k & a_k \end{pmatrix}  \,.
		\end{matrix} 
	\end{equation}
	
	\begin{definition}
		The \emph{parabolic representation scheme} of $K$ is defined as 
		$$ 
            \calR(K) := \mathrm{Spec} \left( A^{\otimes n} / I \right).
        $$
		It carries an $\SL_2(\BC)$-action induced by simultaneous conjugation of the matrix variables. The \emph{parabolic character scheme} is then defined as the GIT quotient
		$$ 
            \calX(K) :=  \mathrm{Spec} \left(\left( A^{\otimes n} / I \right)^{\mathrm{SL}_2(\BC)} \right).
        $$		
		We denote by $\pi :  \calR(K) \rightarrow \calX(K)$ the canonical surjective morphism.
	\end{definition}
	
	\subsection{The parabolic quandle}
	
	The parabolic quandle provides a convenient way to encode trace-$2$ matrices by two-dimensional vectors. More precisely, consider the map
	\begin{equation} \label{eqn:m}
		\binom{x}{y} \longmapsto
		\begin{pmatrix}
			1-xy & x^2 \\
			-y^2 & 1+xy
		\end{pmatrix},
	\end{equation}
	which is surjective onto the set of trace-$2$ matrices. This map is $2$-to-$1$ away from the origin, since the vectors $(x,y)$ and $-(x,y)$ have the same image. For a two-dimensional vector $\alpha$ we denote by $[\alpha]$ the trace-$2$ matrix given in~\eqref{eqn:m}. Using the standard quandle notation,  define
	\begin{equation*}
		\alpha \triangleright \beta:=[\beta]^{-1} \alpha, \qquad
		\alpha\triangleright^{-1}\beta:=[\beta]\alpha \qquad \text{for } \alpha,\beta\in\BC^2.
	\end{equation*}
	A direct computation shows that
	\begin{equation*}
		[g \alpha]=g[\alpha]g^{-1} \qquad
		\text{for } \alpha\in\BC^2\text{ and }g\in\SL_2(\BC)\,.
	\end{equation*}
	It follows that $P:=(\BC^2,\triangleright)$ is a quandle, that is, it satisfies
	$$
	   \alpha \triangleright \alpha= \alpha, \qquad
	   (\alpha \triangleright \beta)\triangleright^{-1}\beta=\alpha, \qquad
      (\alpha\triangleright \beta)\triangleright \gamma = (\alpha\triangleright \gamma)\triangleright( \beta\triangleright \gamma).
	$$
	We call $P$ the \emph{parabolic quandle}. It is equipped with the antisymmetric bilinear form
	$$
	   \langle\cdot,\cdot\rangle\colon P\times P\longrightarrow\BC, \qquad \langle \alpha,\beta \rangle:=\det( \alpha,\beta),
	$$
	with respect to which the quandle operation takes the form
	$$
	   \alpha\triangleright^{\pm1}\beta = \alpha \pm\langle \alpha,\beta\rangle \beta \qquad \text{for } \alpha,\beta \in P.
	$$
	
	At the level of $\SL_2(\BC)$-representations, the parabolic quandle replaces trace-$2$ matrices with 2-dimensional vectors. This yields an \emph{arc-coloring}: an assignment of a quandle element $\alpha_i\in P$ to each arc of $D$ satisfying the following relation at every crossing.
	\begin{equation*}
		\begin{matrix}
			\vcenter{\hbox{\begin{tikzpicture}
						\draw[-stealth,thick] (0,0) -- (1,1);
						\draw[line width=5pt,white]  (0,1)--(1,0);
						\draw[-stealth,thick]  (0,1)--(1,0) ;
						\node at (1,1) [xshift=1.5ex] {$j$};
						\node at (0,0) [xshift=-1ex] {$i$};
						\node at (1,0)  [xshift=1.5ex] {$k$};
					\end{tikzpicture}
			}} \longrightarrow
			\pm \alpha_j=\alpha_i \triangleright \alpha_k ,& \qquad		
			\vcenter{\hbox{\begin{tikzpicture}
						\draw[-stealth,thick]  (0,1)--(1,0) ;
						\draw[line width=5pt,white]  (0,0)--(1,1);
						\draw[-stealth,thick] (0,0) -- (1,1);
						\node at (1,1) [xshift=1.5ex] {$k$};
						\node at (0,1) [xshift=-1ex] {$i$};
						\node at (1,0)  [xshift=1.5ex] {$j$};
					\end{tikzpicture}
			}} \longrightarrow
			\pm \alpha_j=\alpha_i \triangleright^{-1}  \alpha_k	\,.
		\end{matrix}
	\end{equation*}
    Under the map in~\eqref{eqn:m}, these relations are equivalent to the Wirtinger relations in~\eqref{eqn.wirtinger}. However, they involve a sign ambiguity, because the map in~\eqref{eqn:m} is 2-fold.
	We refer to $\alpha_i$ as the \emph{arc color} of the corresponding arc and write an arc-coloring as the tuple $\boldsymbol{\alpha}=(\alpha_1,\ldots,\alpha_n)$ of its arc colors.
		
	At the level of rings, the ring $A$ is replaced with $B=\BC[x,y]$ through the homomorphism
	$\varphi\colon A\rightarrow B$ defined by
	\begin{align*}
		\varphi(a)&=1-xy, &
		\varphi(b)&=x^2, &
		\varphi(c)&=-y^2, &
		\varphi(d)&=1+xy.
	\end{align*}
	 The restriction of the induced morphism $\varphi^\ast\colon \mathrm{Spec}(B)\rightarrow\mathrm{Spec}(A)$ to $\BC$-points agrees with the map in~\eqref{eqn:m}. Lemma~\ref{lem.inj} below implies that $\varphi$ is injective and integral. Hence, by the lying-over theorem, $\varphi^\ast$ is surjective. 
     
	\begin{lemma}\label{lem.inj}
		The homomorphism $\varphi$ is injective, and its image is
		\[
		\operatorname{Im}(\varphi) = \mathbb{C}[x^2,xy,y^2].
		\]
		Consequently, $A \cong \mathbb{C}[x^2,xy,y^2] \hookrightarrow \mathbb{C}[x,y]=B$.
	\end{lemma}
	
	\begin{proof}
		Set $u=1-a$, $v=b$, and $w=-c$.
		Since $d=2-a=1+u$, the relation $ad-bc=1$ becomes $u^2=vw$.
		Hence $A \cong \mathbb{C}[u,v,w]/(u^2-vw)$.
		The generators $u,v,w$ are mapped to $xy, x^2,y^2$ under $\varphi$, respectively.
		Thus the image of $\varphi$ is $\mathbb{C}[x^2,xy,y^2]$.
		
		Every element of $\mathbb{C}[u,v,w]/(u^2-vw)$ can be written in the form $F(v,w)+uG(v,w)$.
		If its image under $\varphi$ is zero, then
		\[
		  F(x^2,y^2)+xy\,G(x^2,y^2)=0.
		\]
		The first summand consists only of monomials $x^{2i}y^{2j}$, whereas the
		second consists only of monomials $x^{2i+1}y^{2j+1}$. Therefore, they
		cannot cancel each other, and hence $F=G=0$. It follows that $\varphi$ is injective.
	\end{proof}
	
	The coordinate ring $A^{\otimes n}/I$ of the scheme $\calR(K)$ is accordingly replaced by
	$$
	   B^{\otimes n}\otimes_{A^{\otimes n}}(A^{\otimes n}/I) \cong B^{\otimes n}/IB^{\otimes n},
	$$
	as shown in the following diagram:
	\begin{equation} \label{eqn.diagram}	
		\begin{tikzcd}
			A^{\otimes n}
			\arrow[r, "\varphi^{\otimes n}"]
			\arrow[d, two heads]
			&
			B^{\otimes n}
			\arrow[d, two heads]
			\\
			A^{\otimes n}/I
			\arrow[r]
			&
			B^{\otimes n}\otimes_{A^{\otimes n}}(A^{\otimes n}/I)
			\cong B^{\otimes n}/IB^{\otimes n}.
		\end{tikzcd}
	\end{equation}

	\begin{definition} 
        The \emph{parabolic quandle scheme} of $D$ is defined as $$\calQ(D) := \Spec (B^{\otimes n}/J)$$
		where $J := I B^{\otimes n} \subset B^{\otimes n}$.
	\end{definition}
    
	Let $\rho$ be the morphism induced from the bottom horizontal map in~\eqref{eqn.diagram}:
	\begin{equation} \label{eqn.rho}
		\rho : \calQ(D) \longrightarrow \calR(K)	\,.
	\end{equation}
	Since $\varphi^*$ is finite and surjective, its $n$-fold product is also finite and surjective. As $\rho$ is obtained from this morphism by base change along $\mathcal R(K)\to\operatorname{Spec}(A^{\otimes n})$, the morphism $\rho$ is finite and surjective. 
    For a point $\mathfrak{p}\in\calQ(D)$, we often identify its image $\rho(\mathfrak{p})$ under $\rho$ with the representation $\rho_\mathfrak{p}$ obtained by evaluation at $\mathfrak{p}$
    \begin{equation}\label{eqn:rhop}
    \rho_{\mathfrak p}: \calG(K)\longrightarrow\SL_2  (\kappa(\mathfrak p)), \qquad
    g_i\longmapsto[\alpha_i(\mathfrak p)].
    \end{equation}
    Here $\kappa(\mathfrak p)$ denotes the residue field at $\mathfrak p$. In particular, if $\mathfrak{p}$ is a $\BC$-point, then $\rho_{\mathfrak{p}}$ is a parabolic $\SL_2(\BC)$-representation.

    \subsection{Nontrivial loci} \label{sec.nontri}
    Recall that the map in~\eqref{eqn:m} is 2-to-1 away from the origin. To obtain a genuine 2-to-1 correspondence, we exclude the identity matrix and the zero vector as follows.
	
	\begin{lemma}\label{lem.torsor}
		Let $\mathfrak{m}=(a-1,\,b,\,c,\, d-1)$
		be the maximal ideal of $A$ corresponding to the identity matrix. Then the morphism $\varphi^\ast : \Spec(B) \rightarrow \Spec(A)$ restricts to a $\mu_2$-torsor
		\[
		      \varphi^\ast : \Spec(B)\setminus\{(x,y)\} \longrightarrow \Spec(A)\setminus\{\mathfrak{m}\},
		\]
		where $\mu_2$ acts diagonally: $\delta\cdot(x,y)=(\delta x,\delta y)$ for $\delta\in\mu_2$.
	\end{lemma}
	
	\begin{proof}
		Since $d=2-a$, the relation $ad-bc=1$ is equivalent to $(1-a)^2=-bc$.
		It follows that $D(b)$ and $D(c)$ form an open cover of $\mathrm{Spec}(A)\setminus \{ \mathfrak{m}\}$. In addition, we have
		\[
    		(\varphi^\ast)^{-1}(D(b))=D(x),	\qquad 	(\varphi^\ast)^{-1}(D(c))=D(y).
		\]
		where $\mathrm{Spec}(B)\setminus\{(x,y)\} = D(x)\cup D(y)$.
	
		Over $D(b)$, consider the $A_b$-algebra isomorphism
		$A_b[T]/(T^2-b) \xrightarrow{\sim}	B_x$
		given by
		\[
		a\mapsto 1-xy,\qquad
		b\mapsto x^2,\qquad
		c\mapsto -y^2,\qquad
		d\mapsto 1+xy,\qquad
		T\mapsto x.
		\]
		Its inverse is given by $x \mapsto T$ and $y \mapsto (1-a)b^{-1}T$. Put $C_b:=A_b[T]/(T^2-b)$.
		The morphism
		\[
		      p: \Spec(C_b)\longrightarrow\Spec(A_b)
		\]
		is finite \'{e}tale and surjective, since $C_b$ is free of rank two over $A_b$ and $2T$ is invertible in $C_b$. To show that $p$ is locally trivial, we consider its pullback along $p$ itself. The resulting morphism is
		\[
		      \Spec(C_b)\times_{\Spec(A_b)}\Spec(C_b) \longrightarrow \Spec(C_b),
		\]
		whose total space has coordinate ring $C_b\otimes_{A_b}C_b$. Writing the generator of the second copy of $C_b$ as $S$, we have
		\[
		      C_b\otimes_{A_b}C_b \cong C_b[S]/(S^2-b) = C_b[S]/(S^2-T^2).
		\]
		Since $T$ is invertible, setting $\delta = ST^{-1}$ gives $C_b\otimes_{A_b}C_b \cong C_b[\delta]/(\delta^2-1)$.
		Hence
		\[
		      \Spec(C_b)\times_{\Spec(A_b)}\Spec(C_b) \cong \mu_2\times\Spec(C_b)
		\]
		This proves the local triviality of $\varphi^\ast$ over $D(b)$. The same argument applies over $D(c)$. 
	\end{proof}
	
	Accordingly, we exclude the trivial character $\chi_{\mathrm{triv}} \in \calX(K)$, sending every element of $\calG(K)$ to $2$, and work with the open subschemes
	\begin{align*}
		\calX^\times(K)	&:=	\calX(K)\setminus{\chi_{\mathrm{triv}}}, \\
		\calR^\times(K)	&:=	\calR(K)\setminus \pi^{-1}(\chi_{\mathrm{triv}}), \\
		\calQ^\times(D)	&:=	\calQ(D)\setminus (\pi\circ\rho)^{-1}(\chi_{\mathrm{triv}}).
	\end{align*}
	By abuse of notation, we use the same symbols for the restricted surjections
	$$
	\pi\colon\calR^\times(K)\longrightarrow\calX^\times(K), \qquad
	\rho\colon\calQ^\times(D)\longrightarrow\calR^\times(K).
	$$
	\begin{lemma} \label{lem:zerocol}Let $\mathfrak p\in\calQ^\times(D)$. Then no arc color vanishes at $\mathfrak p$. That is, if the arc color of the $i$-th arc is written as $\alpha_i=\binom{x_i}{y_i}$, then $x_i$ and $y_i$ cannot both vanish at $\mathfrak p$ for any $i$.
	\end{lemma}
	
	\begin{proof}
		Recall that $\mathfrak p$ determines the representation
		$\rho_{\mathfrak p}$ as in~\eqref{eqn:rhop}.
		Suppose that $\alpha_i(\mathfrak p)=0$ for some $i$. Then 
		$ \rho_{\mathfrak p}(g_i)=I$.
		At a crossing adjacent to the $i$-th arc, the Wirtinger
		relation is of the form $g_j=g_k^{\mp 1}g_i g_k^{\pm1}$.
		It follows that $\rho_{\mathfrak p}(g_j)=I$. Continuing along the  diagram, we conclude that $\rho_{\mathfrak p}(g_j)=I$ for all $j$.
		If we write $\alpha_j(\mathfrak p)=\binom{s}{t}$, then
		\[
		[\alpha_j(\mathfrak p)]
		=
		\begin{pmatrix}
			1-st & s^2\\
			-t^2 & 1+st
		\end{pmatrix}.
		\]
		Thus $[\alpha_j(\mathfrak p)]=I$ implies $s^2=t^2=0$. Since
		they are in the residue field $\kappa(\mathfrak p)$, we have $s=t=0$. Hence all arc
		colors vanish at $\mathfrak p$, and $\rho_{\mathfrak p}$ is the trivial
		representation.
		Therefore, $(\pi\circ\rho)(\mathfrak p)=\chi_{\mathrm{triv}}$,
		contradicting $\mathfrak p\in\calQ^\times(D)$.
	\end{proof}
	
	
	\section{Sign-refined parabolic quandles}\label{sec.quandle}
	
	The relations defining an arc-coloring in the previous section have a sign ambiguity. In this section, we eliminate this ambiguity by assigning a sign to each crossing, and show that these sign data are related to the obstruction class of a parabolic representation.

	\subsection{Sign-refined colorings}
	
	Let $D$ be a diagram of a knot $K$ and let $n$ be the number of crossings of $D$. 
	For simplicity, we fix an ordering of the crossings, so that an assignment of signs to the crossings can be represented by an $n$-tuple of signs.
	
	Let $\boldsymbol{\epsilon} =(\epsilon_1,\dots,\epsilon_n)\in \{ \pm 1\}^n$ be a tuple of signs, where $\epsilon_\ell$ is assigned to the $\ell$-th crossing of $D$.
		An \emph{$\boldsymbol{\epsilon}$-refined arc-coloring} of $D$ is an assignment of a parabolic quandle element $\alpha_i \in P$ to each arc of $D$ satisfying the following relation at each crossing.
		\begin{equation*}
			\vcenter{\hbox{\begin{tikzpicture}
						\draw[-stealth,thick] (0,0) -- (1,1);
						\draw[line width=5pt,white]  (0,1)--(1,0);
						\draw[-stealth,thick]  (0,1)--(1,0) ;
						\node at (0.5,0.5) [above, yshift=0.4ex] {$\ell$};
						\node at (1,1) [xshift=1.5ex] {$j$};
						\node at (0,0) [xshift=-1ex] {$i$};
						\node at (1,0)  [xshift=1.5ex] {$k$};
					\end{tikzpicture}
			}} \longrightarrow
			\epsilon_\ell \, \alpha_j=\alpha_i \triangleright \alpha_k, \qquad
			\vcenter{\hbox{\begin{tikzpicture}
						\draw[-stealth,thick]  (0,1)--(1,0) ;
						\draw[line width=5pt,white]  (0,0)--(1,1);
						\draw[-stealth,thick] (0,0) -- (1,1);
						\node at (0.5,0.5) [above, yshift=0.4ex] {$\ell$};
						\node at (1,1) [xshift=1.5ex] {$k$};
						\node at (0,1) [xshift=-1ex] {$i$};
						\node at (1,0)  [xshift=1.5ex] {$j$};
					\end{tikzpicture}
			}} \longrightarrow
			\epsilon_\ell \, \alpha_j=\alpha_i \triangleright^{-1} \alpha_k.
		\end{equation*}	
	Note that the equation $\epsilon_\ell \, \alpha_j = \alpha_i \triangleright^{\pm 1} \alpha_k$ is equivalent to $\epsilon_\ell \, \alpha_i = \alpha_j \triangleright^{\mp 1} \alpha_k$. Therefore, at each crossing, the color of the over-arc together with  one of the two under-arc colors determines the remaining under-arc color.
	
	\begin{definition}
		The \emph{${\boldsymbol{\epsilon}}$-refined parabolic quandle scheme} is defined as 
		$$
		\calQ_{\boldsymbol{\epsilon}}(D):=
		\Spec \left(B^{\otimes n}/J_{\boldsymbol{\epsilon}}\right)
		$$
		where $J_{\boldsymbol{\epsilon}}$ is the ideal of $B^{\otimes n}$ generated by the defining equations for $\boldsymbol{\epsilon}$-refined arc-colorings. 
	\end{definition}
	
	It is clear that $J\subset J_{\boldsymbol{\epsilon}}$ and hence $\calQ_{\boldsymbol{\epsilon}}(D)$ is a closed subscheme of $\calQ(D)$. In particular, the morphism $\rho$ in~\eqref{eqn.rho} restricts to
	\begin{equation} \label{eqn.res}
		\rho_{\boldsymbol{\epsilon}} : \calQ_{\boldsymbol{\epsilon}}(D) \longrightarrow \calR(K)\, .
	\end{equation}
	As in Section~\ref{sec.nontri}, put $\calQ^\times_{\boldsymbol{\epsilon}}(D):=\calQ_{\boldsymbol{\epsilon}}(D)\setminus (\pi\circ\rho_{\boldsymbol{\epsilon}})^{-1}(\chi_\mathrm{triv})$.
	\begin{proposition}\label{prop:decomp} 
		The scheme $\calQ^\times(D)$ decomposes into a disjoint union of $\calQ^\times_{\boldsymbol{\epsilon}}(D)$:
		\begin{equation*}
			\calQ^\times(D) = \bigsqcup_{\boldsymbol{\epsilon}\in\{\pm1\}^n}
			\calQ_{\boldsymbol{\epsilon}}^\times(D).
		\end{equation*}
		In particular, each
		$\calQ_{\boldsymbol{\epsilon}}^\times(D)$
		is an open-and-closed subscheme of $\calQ^\times(D)$.
	\end{proposition}
	
	\begin{proof}
		Let $J_\ell \subset B^{\otimes n}$ be the ideal generated by the corresponding Wirtinger relation at the $\ell$-th crossing, and let
		$J_{\ell,+}$ and $J_{\ell,-}$ be the ideals generated by two sign-refined arc-coloring relations at the same crossing. Thus,
		$$
		J=\sum_{\ell=1}^n J_\ell,\qquad
		J_{\boldsymbol{\epsilon}}=\sum_{\ell=1}^n J_{\ell,\epsilon_\ell}.
		$$
		Writing the Wirtinger relation at the crossing as
		$[u]=[v]$ for $u=\binom{x}{y}$, $v=\binom{s}{t}$, one has
		$$
		J_{\ell}=(x^2-s^2,\,xy-st,\,y^2-t^2), \quad 
		J_{\ell,+}=(x - s,\,y- t), \quad	J_{\ell,-}=(x + s,\,y+ t).
		$$
		On $D(x)$, put $e=s/x$. Then $e^2=1$ and $v=e u$ modulo $J_{\ell}$.
		Hence
		$$
		J_{\ell,+}/J_{\ell}=(e-1),	\qquad	J_{\ell,-}/J_{\ell}=(e+1).
		$$
		Since $2$ is invertible, the ideals $(e-1)$ and $(e+1)$ are coprime, and their product is zero. It follows that on $D(x)$, we have
		$$
		(J_{\ell})_x=(J_{\ell,+})_x \cap (J_{\ell,-})_x, \qquad
		(J_{\ell,+})_x +(J_{\ell,-})_x=(B^{\otimes n})_x.
		$$
		The same argument applies on $D(y)$, hence these identities hold wherever $u\neq0$.
		
		Let $\mathfrak p$ be a point of $\mathcal Q^\times(D)$.
		Lemma~\ref{lem:zerocol} shows that no arc color vanishes at $\mathfrak p$, hence  
		$$
		(J_\ell)_{\mathfrak p} = (J_{\ell,+})_{\mathfrak p} \cap (J_{\ell,-})_{\mathfrak p}, \qquad
		(J_{\ell,+})_{\mathfrak p}	+ (J_{\ell,-})_{\mathfrak p} = (B^{\otimes n})_{\mathfrak p}.
		$$
		for all $\ell$.	We use the identity
		$$ L+(M\cap N) = (L+M)\cap(L+N)	$$
		which holds whenever $M$, $N$ are coprime. Applying this identity repeatedly, we deduce that
		$$
		J_{\mathfrak p} = \bigcap_{\boldsymbol{\epsilon}\in\{\pm1\}^n} (J_{\boldsymbol{\epsilon}})_{\mathfrak p} \quad \textrm{for every } \mathfrak{p} \in \calQ^\times(D).
		$$
		Therefore, the subschemes $\calQ_{\boldsymbol{\epsilon}}^\times(D)$ cover $\calQ^\times(D)$.
		
		If $\boldsymbol{\epsilon}\neq\boldsymbol{\epsilon}'$, choose any $\ell$ such that $\epsilon_\ell\neq\epsilon_\ell'$. Then the two ideals $J_{\ell,{\epsilon_\ell}}$ and $J_{\ell,{\epsilon'_\ell}}$ are coprime on $\calQ^\times(D)$. Since they are contained in $J_{\boldsymbol{\epsilon}}$ and $J_{\boldsymbol{\epsilon}'}$, respectively, the restrictions of $J_{\boldsymbol{\epsilon}}$ and
		$J_{\boldsymbol{\epsilon}'}$ to $\calQ^\times(D)$ are also coprime. It follows that
		\[
		\calQ_{\boldsymbol{\epsilon}}^\times(D) \cap \calQ_{\boldsymbol{\epsilon}'}^\times(D) = \varnothing \,.
		\]
		This proves the asserted disjoint decomposition.
	\end{proof}

	For $\boldsymbol{\epsilon} =(\epsilon_1,\dots,\epsilon_n)\in \{ \pm 1\}^n$, we define its \emph{total sign} by
	\[
	\mathrm{sgn}(\boldsymbol{\epsilon}) = \prod_{\ell=1}^n \epsilon_\ell \in \{\pm1\}.
	\]
	
	\begin{proposition} \label{prop:Qebijection}
		Let	$\boldsymbol{\epsilon},\boldsymbol{\epsilon}' \in\{\pm1\}^n$ satisfy $\mathrm{sgn}(\boldsymbol{\epsilon})	= \mathrm{sgn}(\boldsymbol{\epsilon}')$. Then there is an isomorphism 
		\begin{equation} \label{eqn:PHI}
			\Psi : \mathcal Q_{\boldsymbol{\epsilon}}(D) \xrightarrow{\sim}
			\mathcal Q_{\boldsymbol{\epsilon}'}(D)
		\end{equation}
		such that $\rho_{\boldsymbol{\epsilon}'} \circ \Psi = \rho_{\boldsymbol{\epsilon}}$.
		In particular, it restricts to an isomorphism $\mathcal Q_{\boldsymbol{\epsilon}}^\times(D) \xrightarrow{\sim}
		\mathcal Q_{\boldsymbol{\epsilon}'}^\times(D)$.
	\end{proposition}
	
	\begin{proof}
		Label the arcs and crossings of $D$ cyclically along the diagram so that the $k$-th crossing lies between the $k$-th and $(k+1)$-st arcs. Let $i_k$ be the index of the over-arc at the $k$-th crossing (see Figure~\ref{fig:indexalong}). Then, the $\boldsymbol{\epsilon}$-refined arc-coloring relation at the $k$-th crossing is of the form $\epsilon_k\alpha_{k+1} = \alpha_k \pm \langle \alpha_k, \alpha_{i_k} \rangle \alpha_{i_k}$.
        
		We define signs $\eta_k \in\{\pm1\}$ recursively by
		\[
		\eta_1=1, \qquad \eta_{k+1} 	= \eta_k \frac{\epsilon_k}{\epsilon_k'}	\quad	\text{for }	1\leq k\leq n.
		\]
		Since $\mathrm{sgn}(\boldsymbol{\epsilon})	= \mathrm{sgn}(\boldsymbol{\epsilon}')$, we have
		$\eta_{n+1} = \eta_1 \prod_{k=1}^n\frac{\epsilon_k}{\epsilon'_k} = \eta_1$, 
		so this definition is compatible with the cyclic labeling. Consider the involutive automorphism
		\[
		      \widetilde{\Psi}:  \BA^{2n} \rightarrow \BA^{2n}, \qquad  (\alpha_1,\ldots,\alpha_n) \longmapsto (\eta_1 \alpha_1,\ldots,\eta_n \alpha_n).
		\]
		Then for 
        \begin{align*}
            F_k &= 	\epsilon_k \alpha_{k+1}	-	\alpha_k 	\mp 	\langle \alpha_k, \alpha_{i_k} 	\rangle\alpha_{i_k}, & 
            F_k' &=	\epsilon'_k \alpha_{k+1}	-	\alpha_k 	\mp 	\langle \alpha_k, \alpha_{i_k} 	\rangle\alpha_{i_k},
        \end{align*} 
		we have
		  \[
		      \widetilde{\Psi}^\ast(F_k') = \epsilon_k'\eta_{k+1} \alpha_{k+1} - \eta_k \alpha_k \mp \langle \eta_k \alpha_k, \eta_{i_k} \alpha_{i_k} \rangle \eta_{i_k}\alpha_{i_k} = \eta_k F_k.
		\]
		It follows that $\widetilde{\Psi}^\ast (J_{\boldsymbol{\epsilon}'}) = J_{\boldsymbol{\epsilon}}$. Therefore, we obtain the induced isomorphism $\Psi$ as in~\eqref{eqn:PHI}.  Since $\eta_k = \pm1$ for all $k$, it is obvious that 
		$ \rho_{\boldsymbol{\epsilon}'} \circ \Psi = \rho_{\boldsymbol{\epsilon}}$.
	\end{proof}
		
	Propositions~\ref{prop:decomp} and~\ref{prop:Qebijection} imply that not all $\calQ^\times_{\boldsymbol{\epsilon}}(D)$ are needed to cover $\calR^\times(K)$. Indeed, for any two $\boldsymbol{\epsilon},\boldsymbol{\epsilon}' \in\{\pm1\}^n$ with $\mathrm{sgn}(\boldsymbol{\epsilon}) \neq \mathrm{sgn}(\boldsymbol{\epsilon}')$, the morphism
	\begin{equation} \label{eqn.surj}
		\rho_{\boldsymbol{\epsilon}} \sqcup \rho_{\boldsymbol{\epsilon}'} : \calQ^\times_{\boldsymbol{\epsilon}}(D) \sqcup 	\calQ^\times_{\boldsymbol{\epsilon}'}(D) \longrightarrow \calR^\times(K)
	\end{equation}
	is surjective.
	
	\subsection{The obstruction class}\label{sec:obs}
    
	A parabolic $\SL_2(\BC)$-representation $\rho$ satisfies $\Tr(\rho(\mu))=2$ for a meridian $\mu$. Since the canonical longitude $\lambda$ commutes with $\mu$, we have
	\[
	   \Tr(\rho(\lambda)) = \pm 2\,.
	\]
	This sign can be interpreted as the obstruction to $\rho$ being boundary-unipotent, and can be regarded as a class in the second relative cohomology group of the knot exterior with coefficients in $\{\pm 1\}$.  We thus refer to it as the \emph{obstruction class} of $\rho$. 
	
	According to the obstruction class, the  parabolic representation scheme decomposes into two open-and-closed subschemes:
	\[
	   \calR^\times(K) = \calR^\times_+(K) \sqcup \calR^\times_-(K).
	\]
	More precisely, the two subschemes can be defined as follows.
	\begin{lemma}\label{lem:obs} 
        Let $\rho_{\textrm{univ}} : \calG(K) \rightarrow \mathrm{SL}_2(\Gamma (\calR^\times(K)))$ be the tautological representation, and 
		\[
		      \tau_\lambda := \mathrm{Tr} (\rho_{\mathrm{univ}}(\lambda)) \in \Gamma (\calR^\times(K)).
		\]
		Then $(\tau_\lambda-2)(\tau_\lambda+2)=0$. Consequently, the two open-and-closed subschemes $\calR_\pm^\times(K)$ are cut out by the equations $\tau_\lambda = \pm 2$.
	\end{lemma}
	
	\begin{proof}
		For simplicity,  write
		\[
		      M:=\rho_{\mathrm{univ}}(\mu) =
    		\begin{pmatrix}
    			1+a & b\\
    			c & 1-a
    		\end{pmatrix}, \qquad
    		L:=\rho_{\mathrm{univ}}(\lambda) =
    		\begin{pmatrix}
    			p&q\\
    			r&s
    		\end{pmatrix}.
		\]
		Since $\det(M)=1$, we have $a^2+bc=0$. If $a=b=c=0$ at a point of $\calR^\times(K)$, then the corresponding representation sends $\mu$ to the identity, and thus is trivial, contradicting the definition of $\calR^\times(K)$. Therefore, $D(a)$, $D(b)$, $D(c)$ form an open cover of $\calR^\times(K)$.
		
		Since $\lambda$ commutes with $\mu$, we have $ML=LM$. Comparing matrix entries gives
		\[
		      br=cq, \qquad b(p-s)=2aq, \qquad c(p-s)=2ar.
		\]
		On $D(b)$, we have $r = cq/b$ and $p-s = 2aq/b$. Setting $t= (p+s)/2$, we have
		\[
		      p=t+aq/b, \qquad s=t-aq/b.
		\]
		Since $\det(L)=1$, we have 
		\begin{align*}
			1 =ps-qr =t^2-\frac{a^2q^2}{b^2}-\frac{cq^2}{b}	=t^2.
		\end{align*}
		It follows that $\Tr(L)^2=(p+s)^2=4$  on $D(b)$. Similar arguments work on $D(a)$ and $D(c)$. This proves that $\tau_\lambda^2=4$ on $\calR^\times(K)$.  Moreover, since the ideals $(\tau_\lambda-2)$ and
		$(\tau_\lambda+2)$ are coprime, the asserted open-and-closed decomposition follows.
	\end{proof}

	\begin{theorem} \label{thm:obs}
		Fix $\boldsymbol{\epsilon}\in\{\pm1\}^n$ and let $s=\sgn(\boldsymbol{\epsilon})$ denote its total sign. Then  the morphism $\rho_{\boldsymbol{\epsilon}} : \calQ_{\boldsymbol{\epsilon}}^\times(D) \rightarrow	\calR^\times(K)$ takes values in $\calR_s^\times(K)$. Thus, it induces a morphism
		\begin{equation} \label{eqn:rhoe}
			\rho_{\boldsymbol{\epsilon}} : \calQ_{\boldsymbol{\epsilon}}^\times(D) \longrightarrow \calR_s^\times(K).
		\end{equation}
		By abuse of notation, we denote the induced morphism by the same symbol $\rho_{\boldsymbol{\epsilon}}$.
	\end{theorem}
	\begin{proof}
		Let $ \rho_\calQ : \calG(K) \rightarrow \SL_2(\Gamma(Q_{\boldsymbol{\epsilon}}^\times(D)))$ be the pullback of the tautological representation $\rho_\mathrm{univ}$ in Lemma~\ref{lem:obs} along $\rho_{\boldsymbol{\epsilon}}:\calQ_{\boldsymbol{\epsilon}}^\times(D) \rightarrow \calR^\times(K)$. In what follows, we suppress $K$ and $D$ from the notation for simplicity.
		
		\begin{figure}[!h]
			\centering
			\scalebox{1}{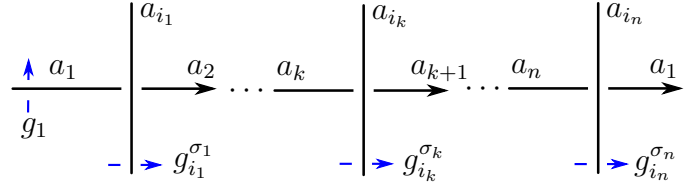}
			\caption{Arc indices along the diagram.}
			\label{fig:indexalong}
		\end{figure}
		
		Label the arcs and crossings cyclically along the diagram so that the $k$-th crossing lies between the $k$-th and $(k+1)$-st arcs, and let $i_k$ be the index of the over-arc at the $k$-th crossing. We denote the arc-color of the $k$-th arc on $\calQ_{\boldsymbol{\epsilon}}^\times$ by
		\[
		\alpha_k= 
        \begin{pmatrix}
			x_k\\
			y_k
		\end{pmatrix} \in
		\Gamma(\calQ_{\boldsymbol{\epsilon}}^\times) \oplus \Gamma(\calQ_{\boldsymbol{\epsilon}}^\times).
		\]
		with $\alpha_{n+1}=\alpha_1$. If $\sigma_k\in\{\pm1\}$ denotes the sign of the $k$-th crossing, then the defining equation of $\calQ_{\boldsymbol{\epsilon}}$ at the $k$-th crossing is written as
		\[
		      \epsilon_k\alpha_{k+1} = [\alpha_{i_k}]^{-\sigma_k}\alpha_k.
		\]
		Iterating these equations along the diagram, we obtain
		\[
		      s \, \alpha_1 = \epsilon_1 \cdots \epsilon_n  \alpha_1 = [\alpha_{i_n}]^{-\sigma_n}\cdots [\alpha_{i_1}]^{-\sigma_1}\alpha_1.
		\]
		
		On the other hand, if we choose $g_1$ as a meridian, then the blackboard-framed longitude is
		$\lambda^\circ = g_{i_1}^{\sigma_1}\cdots g_{i_n}^{\sigma_n}$. See Figure~\ref{fig:indexalong}. It follows that  $\rho_Q(\lambda^\circ) = [\alpha_{i_1}]^{\sigma_1}\cdots [\alpha_{i_n}]^{\sigma_n}$ and thus
		\begin{equation} \label{eqn.signdp}
			\rho_Q(\lambda^\circ)\alpha_1= s\, \alpha_1 \,.
		\end{equation}
		As no arc-color vanishes at any point of $\calQ_{\boldsymbol{\epsilon}}^\times$ (see Lemma~\ref{lem:zerocol}), we have $\calQ_{\boldsymbol{\epsilon}}^\times=D(x_1)\cup D(y_1)$. On $D(x_1)$, consider the matrix
		\[
		P=
		\begin{pmatrix}
			x_1&0\\
			y_1&x_1^{-1}
		\end{pmatrix}
		\in\SL_2(\Gamma(D(x_1))),
		\]
		whose first column is $\alpha_1$. With respect to this basis, Equation~\eqref{eqn.signdp} implies that
		\[
		P^{-1} \rho_Q(\lambda^\circ) P =
		\begin{pmatrix}
			s&*\\
			0&d
		\end{pmatrix}
		\]
		for some $d$, which should equal to $s$ by the fact that the determinant is 1. This proves that the trace of  $\rho_Q(\lambda^\circ)$ is equal to $2s$ on $D(x_1)$. We prove the same equality on $D(y_1)$ similarly by using
		\[
		\begin{pmatrix}
			x_1&-y_1^{-1}\\
			y_1&0
		\end{pmatrix}
		\in\SL_2 (\Gamma(D(y_1))).
		\]
		Since these two open subsets cover $\calQ_{\boldsymbol{\epsilon}}^\times$, the equality holds in $\Gamma(\calQ_{\boldsymbol{\epsilon}}^\times)$.
		
		The canonical longitude differs from the blackboard-framed longitude by a power of the meridian $g_1$. Using the conjugation by the same matrix $P$, one easily checks that the trace of $\rho_Q(\lambda)$ is also $2s$.
		This proves that $\rho_{\boldsymbol{\epsilon}} : \calQ_{\boldsymbol{\epsilon}}^\times \rightarrow \calR^\times$
		takes values in $\calR^\times_s$, cut out by $\tau_\lambda =2s$.
	\end{proof}
	
	\begin{theorem} \label{prop.mu2torsor}
		The morphism $	\rho_{\boldsymbol{\epsilon}}:\calQ_{\boldsymbol{\epsilon}}^\times(D) \rightarrow
		\calR_s^\times(K)$ is a surjective $\mu_2$-torsor.
	\end{theorem}
	\begin{proof}
		The surjectivity is clear from Theorem~\ref{thm:obs} together with the fact that the morphism in~\eqref{eqn.surj} is surjective. In what follows, we suppress $K$ and $D$ from the notation for simplicity. 

        By Lemma~\ref{lem.torsor}, the morphism
		\[
		      (\Spec(B)\setminus\{(x,y)\})^n \longrightarrow (\Spec(A)\setminus\{\mathfrak m\})^n
		\]
		is a $\mu_2^n$-torsor. As its pullback along $\calR^\times \rightarrow (\Spec(A)\setminus\{\mathfrak m\})^n$ is $\rho:\calQ^\times \rightarrow \calR^\times$, $\rho$ is finite étale. Since
		$\calQ_{\boldsymbol{\epsilon}}^\times$ is an open-and-closed subscheme of $\calQ^\times$ (see Proposition~\ref{prop:decomp}), the morphism $\rho_{\boldsymbol{\epsilon}}$ in~\eqref{eqn:rhoe}
		is also finite étale. Together with its surjectivity, $\rho_{\boldsymbol{\epsilon}}$ is an \'{e}tale cover.
		
		We now consider the pullback of $\rho_{\boldsymbol{\epsilon}}$ along itself and claim that the morphism
		\begin{equation} \label{eqn.trivialization}
			\mu_2\times \calQ_{\boldsymbol{\epsilon}}^\times
			\longrightarrow
			\calQ_{\boldsymbol{\epsilon}}^\times \times_{\calR_s^\times}\calQ_{\boldsymbol{\epsilon}}^\times, \qquad 
            (\delta,\boldsymbol{\alpha}) \mapsto (\boldsymbol{\alpha},\delta\boldsymbol{\alpha})
		\end{equation}
		is an isomorphism. Let $\boldsymbol{\alpha}$ and $\boldsymbol{\alpha}'$ be the arc colors pulled back from the first and second factors of $\calQ_{\boldsymbol{\epsilon}}^\times \times_{\calR_s^\times}
		\calQ_{\boldsymbol{\epsilon}}^\times$, respectively. Since the two arc-colorings have the same
		image in $\calR_s^\times$, we have $[\alpha_i']=[\alpha_i]$ for all $i$. It follows that for each $i$, there exists a unique regular function
		\[
		\delta_i\in\Gamma \left( \calQ_{\boldsymbol{\epsilon}}^\times \times_{\calR_s^\times} \calQ_{\boldsymbol{\epsilon}}^\times \right)
		\]
		such that $\delta_i^2=1$ and $\alpha_i'=\delta_i\alpha_i$.
		At a crossing whose two under-arcs are colored by $\alpha_i$ and $\alpha_j$ and whose over-arc is colored by $\alpha_k$,  the $\boldsymbol{\epsilon}$-refined arc-coloring relations for $\boldsymbol{\alpha}$ and $\boldsymbol{\alpha}'$ are written as
		\[
		\epsilon_\ell\alpha_j =	\alpha_i\triangleright^{\pm1}\alpha_k \qquad \text{and}	\qquad
		\epsilon_\ell\alpha_j' =	\alpha_i'\triangleright^{\pm1}\alpha_k'.
		\]
		Substituting $\alpha_i'=\delta_i\alpha_i$, we obtain
		\begin{align*}
			\epsilon_\ell\delta_j\alpha_j =	(\delta_i\alpha_i)	\triangleright^{\pm1} (\delta_k\alpha_k) 
            = \delta_i (\alpha_i\triangleright^{\pm1}\alpha_k) = \epsilon_\ell\delta_i\alpha_j.
		\end{align*}
		Since $\alpha_j$ is non-zero by Lemma~\ref{lem:zerocol}, we have $\delta_j = \delta_i$. Along the diagram, we deduce that all the $\delta_i$ coincide; denote their common function by $\delta$. It follows that the morphism defined by $(\boldsymbol{\alpha}, \boldsymbol{\alpha}') \mapsto (\delta, \boldsymbol{\alpha})$ is the inverse of the morphism in~\eqref{eqn.trivialization}. Therefore, after pulling back along the étale
		cover $	\rho_{\boldsymbol{\epsilon}}$, it becomes the trivial $\mu_2$-bundle. Hence $\rho_{\boldsymbol{\epsilon}}$ is a $\mu_2$-torsor.
	\end{proof}

	\subsection{Open charts and normalization}
	
	We now consider the parabolic character scheme. As in Section~\ref{sec:obs}, $\calX^\times(K)$ decomposes into two open-and-closed subschemes according to the obstruction class:
	\[
	   \calX^\times(K) = \calX^\times_+(K) \sqcup \calX^\times_-(K).
	\]
    More precisely, the proof of Lemma~\ref{lem:obs} implies that the trace function $t_\lambda$ of the canonical longitude $\lambda$ satisfies $(t_\lambda-2)(t_\lambda+2)=0$ on $\calX^\times(K)$, and $\calX^\times_{\pm}(K)$ are cut out by $t_\lambda = \pm2$. For each sign $s$, the canonical surjection $\pi$
	restricts to a surjective morphism
	$$\pi:\calR_s^\times(K) \rightarrow\calX_s^\times(K)$$
	which we denote by the same symbol $\pi$.
	
	We fix an ordering of the arcs of the diagram $D$, and  let $g_k$ be the Wirtinger generator of $\calG(K)$ associated to the $k$-th arc. Put
	\[
	d_k := t_{g_{1}g_{k}^{-1}}-2 \in \Gamma (\calX(K)).
	\]
	Here $t_g$ denote the trace function of $g\in\calG(K)$.
	
	\begin{lemma}
		\label{lem:xdeocmp}
		Fix a sign $s$, and let $\calX^\times_{s,k}(K) := D(d_k)\cap\calX^\times_s(K)$ for $k=2,\ldots,n$.
		Then
		\[
		\calX^\times_s(K) =	\bigcup_{k=2}^n \calX^\times_{s,k}(K).
		\]
	\end{lemma}
	
	\begin{proof}
		Suppose that the union does not cover $\calX^\times_s(K)$. Its complement is then a non-empty closed subset, and hence contains a closed point $\chi$. We choose a parabolic representation $\rho : \calG(K) \rightarrow\SL_2(\BC)$ whose character is $\chi$. Since $\chi\notin D(d_k)$, we have 
        $ \Tr ( \rho(g_{1})\rho(g_{k})^{-1}) = 2$ for all $k$. Observe that for trace-$2$ matrices $X,Y\in\SL_2(\BC)$,
		\[
		\Tr(XY^{-1})=2 \Longleftrightarrow XY =YX.
		\]
		It follows that every $\rho(g_{k})$ commutes with $\rho(g_{1})$. Hence $\rho$ is reducible. Since $\rho$ is both parabolic and reducible, the Burde--de Rham theorem implies that $\rho$ is abelian. Hence its character $\chi$ is trivial, contradicting $\chi\in\calX^\times_s(K)$.
	\end{proof}
	Pulling back the open cover of $\calX_s^\times(K)$ along the surjective morphism $\pi\circ\rho_{\boldsymbol{\epsilon}}$ for any $\boldsymbol{\epsilon}$ satisfying $\sgn(\boldsymbol{\epsilon})=s$, we obtain an open cover of $\calQ_{\boldsymbol{\epsilon}}^\times(D)$.  Precisely, put
    \begin{align*}
	\calQ^\times_{\boldsymbol{\epsilon},k}(D) &:= (\pi\circ\rho_{\boldsymbol{\epsilon}})^{-1} (\calX_{s,k}^\times(K))
	\subset \calQ_{\boldsymbol{\epsilon}}^\times(D).
    \end{align*}
	\begin{lemma} \label{lem:pullbackchart}
		For $k=2,\ldots,n$, one has
		\begin{align*}
		\calQ^\times_{{\boldsymbol{\epsilon}},k}(D) &= D(\Delta_k),
		\end{align*}
		where $ \Delta_k:=\langle\alpha_{1},\alpha_{k}\rangle \in \Gamma (\calQ_{\boldsymbol{\epsilon}}^\times(D))$.
	\end{lemma}
	
	\begin{proof}
		A direct computation gives
		$$
		      \Tr\bigl([\alpha][\beta]^{-1}\bigr)-2 = \langle\alpha,\beta\rangle^2
		$$
		for any arc colors $\alpha,\beta$. Applying this identity to $\alpha_{1}$ and $\alpha_{k}$, the pullback of $d_k$ along $\pi \circ \rho_{\boldsymbol{\epsilon}}$ is equal to $\Delta_k^2$.
		Therefore, the pullback of $D(d_k)$ is equal to $D(\Delta_k^2)$ and thus to $D(\Delta_k)$.
	\end{proof}
	
	For $k=2,\ldots,n$, we define the \emph{$k$-th normalized scheme} by
	$$
	\calN_{\boldsymbol{\epsilon},k}(D):= \Spec\left(\frac{\BC[x_1,y_1,\ldots,x_n,y_n]} {J_{\boldsymbol{\epsilon}} +(x_{1}-1,y_{1},x_{k})}\right)_{y_{k}}.
	$$
	On $\calN_{\boldsymbol{\epsilon},k}(D)$, one has $\alpha_{1} = \binom{1}{0}$ and $\alpha_{k} = \binom{0}{y_{k}}$.
    Theorem~\ref{thm:normalizedchart} below shows that, for each $k=2,\ldots,n$, the normalized scheme $\calN_{\boldsymbol{\epsilon},k}(D)$ is isomorphic to the open chart $\calX^\times_{s,k}(K)$ of $\calN_{\boldsymbol{\epsilon},k}(D)$. Thus, it suffices to compute $\calN_{\boldsymbol{\epsilon},k}(D)$ for $k=2,\ldots,n$ in order to determine the scheme structure of $\calX^\times_s(K)$.

	\begin{theorem} \label{thm:normalizedchart}
		For $k=2,\ldots,n$, 
		$$
		\Theta_k: \SL_2 \times\calN_{\boldsymbol{\epsilon},k}(D)\longrightarrow \calQ^\times_{{\boldsymbol{\epsilon}},k}(D),\qquad
		(g,\boldsymbol{\alpha})\longmapsto g\boldsymbol{\alpha},
		$$
		is an isomorphism.  
		In addition, it induces an isomorphism
		$$
		\calN_{\boldsymbol{\epsilon},k}(D) \overset{\sim}{\longrightarrow} \calX_{s,k}^\times(K)\,.
		$$
	\end{theorem}
	
	\begin{proof}
		Let	$\boldsymbol{\alpha}\in \calQ^\times_{{\boldsymbol{\epsilon}},k}(D)$ with $
		\alpha_{1}	= \binom{x_1}{y_1}$, $\alpha_{k} =\binom{x_{k}}{y_{k}}$.
		Since $\Delta_k$ is invertible on $\calQ^\times_{{\boldsymbol{\epsilon}},k}(D)$,
		$$
		F(\boldsymbol{\alpha}) 	=
		\begin{pmatrix}
			y_{k}\Delta_k^{-1}	& 	-x_{k}\Delta_k^{-1} \\
			-y_{1}		& 		x_{1}
		\end{pmatrix}
		$$
		defines a regular $\SL_2$-valued function $F$ on $\calQ^\times_{{\boldsymbol{\epsilon}},k}(D)$.
		Then we have
		$$
		F(\boldsymbol{\alpha})\alpha_{1} = \binom{1}{0},	\qquad
		F(\boldsymbol{\alpha})\alpha_{k} = \binom{0}{\Delta_k}.
		$$
		It follows that	$F(\boldsymbol{\alpha})\boldsymbol{\alpha} \in	\calN_{\boldsymbol{\epsilon},k}(D)$. The inverse of $\Theta_k$ is then  given by
		$$
		\boldsymbol{\alpha}	\longmapsto	(F(\boldsymbol{\alpha})^{-1}, \, F(\boldsymbol{\alpha})\boldsymbol{\alpha}).
		$$
		Indeed, if $g\in\SL_2$ satisfies $g\alpha_{1} = \binom{1}{0}$ and $g\alpha_{k} = \binom{0}{c}$ for some $c$,
		then 
		$$
		c =	\det \left(g\alpha_{1},	g\alpha_{k}	\right) = \det \left(\alpha_{1}, \alpha_{k} \right) = \Delta_k.
		$$
		Since $\alpha_1$ and $\alpha_k$ form a basis of
		$\mathcal O({\calQ^\times_{{\boldsymbol{\epsilon}},k}(D)})^{\oplus2 }$,  these conditions determine $g$ uniquely. This proves that $\Theta_k$ is an isomorphism.

		By Proposition~\ref{prop.mu2torsor}, the restriction of $\rho_{\boldsymbol{\epsilon}}: \calQ^\times_{\boldsymbol{\epsilon}}(D) \rightarrow \calR^\times_{s}(K)$ to $\calQ^\times_{{\boldsymbol{\epsilon}},k}(D)$
		$$
		\rho_{\boldsymbol{\epsilon}}: 	\calQ^\times_{{\boldsymbol{\epsilon}},k}(D)\longrightarrow	\calR^\times_{s,k}(K)
		$$
		is a $\mu_2$-torsor. Therefore,
		$\Gamma (\calR^\times_{s,k}(K)) \cong \Gamma (\calQ^\times_{{\boldsymbol{\epsilon}},k}(D))^{\mu_2}$.
		On the other hand,  $\mu_2$ is the center of $\SL_2$ and acts trivially on
		$\calR^\times_{s,k}(K)$. Hence, the $\SL_2$-action factors through
		$\PSL_2=\SL_2/\mu_2$. It follows that
		\[
		\Gamma (\calX_{s,k}^\times(K)) \cong
		\Gamma (\calR^\times_{s,k}(K))^{\PSL_2}	\cong 
        (\Gamma (\calQ^\times_{{\boldsymbol{\epsilon}},k}(D))^{\mu_2})^{\PSL_2} \cong 
        \Gamma (\calQ^\times_{{\boldsymbol{\epsilon}},k}(D))^{\SL_2}.
		\]
		Under the isomorphism $\Theta_k$, the $\SL_2$-action on $\SL_2 \times\calN_{\boldsymbol{\epsilon},k}(D)$ is given by the left multiplication on the first factor. It follows that
		$$
		\Gamma (\calQ^\times_{{\boldsymbol{\epsilon}},k}(D))^{\SL_2} \cong
        \Gamma(\SL_2 \times\calN_{\boldsymbol{\epsilon},k}(D))^{\SL_2} \cong
		\Gamma (\calN_{\boldsymbol{\epsilon},k}(D)).
		$$
		Therefore, we obtain $\calN_{\boldsymbol{\epsilon},k}(D) \cong \calX_{s,k}^\times(K)$.
	\end{proof}
    
    If we restrict our attention to $\BC$-points, there is a more convenient way than computing all the normalized schemes $\calN_{\boldsymbol{\epsilon},k}(D)$. Theorem~\ref{thm:normalizedchart} gives a set-bijection
    $$
    \calQ^\times_{{\boldsymbol{\epsilon},k}}(D)(\BC)/ \SL_2(\BC) \overset{1:1}{\longleftrightarrow} \calX^\times_{s,k}(K)(\BC)
    $$
    for each $k=2,\ldots,n$. Since these subsets cover $\calQ^\times_{{\boldsymbol{\epsilon}}}(D)(\BC)$ and $\calX^\times_{s}(K)(\BC)$, respectively, we obtain a set-bijection
    $$
    \calQ^\times_{{\boldsymbol{\epsilon}}}(D)/ \SL_2(\BC) \overset{1:1}{\longleftrightarrow}  \calX^\times_{s}(K)(\BC).
    $$
    Since a $\BC$-point of $\calX^\times_{s}(K)$ is a non-trivial parabolic character, we call a $\BC$-point of $\calQ^\times_{\boldsymbol{\epsilon}}(D)$ a \emph{non-trivial} $\boldsymbol{\epsilon}$-refined arc-coloring. Two such colorings $\boldsymbol{\alpha}$ and $\boldsymbol{\alpha}'$ are said to be \emph{equivalent} if there exists $g\in\SL_2(\BC)$ such that $g\boldsymbol{\alpha}=\boldsymbol{\alpha}'$; equivalently, if they have the same image in $\calX^\times_{s}(K)(\BC)$ under $\pi\circ\rho_{\boldsymbol{\epsilon}}$.

	\begin{corollary} \label{cor:normalizedrep}
		Each equivalence class of non-trivial $\boldsymbol{\epsilon}$-refined arc-colorings has a unique representative of the form
		$$
		\alpha_{1} = \binom{1}{0}, \quad \alpha_{2} = \binom{\ast}{0},	\quad
		\ldots,	\quad \alpha_{{k-1}} =\binom{\ast}{0}, \quad\alpha_{k} =\binom{0}{\ast}.
		$$
		Here $k$ is the smallest index such that $\Delta_k\neq 0$, and each $\ast$ is a non-zero complex number.
	\end{corollary}
	
	\begin{proof}
		Let $\boldsymbol{\alpha}$ be a non-trivial $\boldsymbol{\epsilon}$-refined arc-coloring. 
        By Lemma~\ref{lem:pullbackchart}, there exists an index $k$ such that $\Delta_k\neq 0$. 
        Choose the smallest such index. By Theorem~\ref{thm:normalizedchart}, the $\SL_2(\BC)$-orbit of $\boldsymbol{\alpha}$ contains a unique	representative satisfying
		$$
		\alpha_{1} 	=\binom{1}{0}, \qquad \alpha_{k} = \binom{0}{\Delta_k}.
		$$
		For each $2 \leq j<k$, the minimality of $k$ gives
        $\Delta_j =	\langle\alpha_{1},\alpha_{j}\rangle	= 0$, that is, $\alpha_{j}$ is proportional to $\alpha_{1}$.
		Since $\alpha_{j}$ cannot be zero, it has the form $\binom{\ast}{0}$ with a non-zero first entry. 
        The uniqueness follows from Theorem~\ref{thm:normalizedchart}. 
	\end{proof}

    \begin{remark}
		The condition that $k$ be the smallest index such that $\Delta_k\neq 0$ is useful for partitioning the $\BC$-points into non-overlapping cases. It should be distinguished from the open cover in Lemma~\ref{lem:xdeocmp}.
		Indeed, over the open chart $\calX^\times_{k,\boldsymbol{\epsilon}}$, the equation $d_k=0$ pulls back to	$\Delta_k^2=0$, not to $\Delta_k=0$. Therefore, to compute scheme-theoretic information, such as the multiplicity, one must work on the open charts.
    \end{remark}
	
	\section{Examples}\label{sec.example}
	
	In this section, we illustrate how the quandle method can be applied in practice for computing parabolic characters through several examples.
	
	\subsection{The knot $4_1$}\label{sec:4_1}
	
    We begin with $K=4_1$ as a toy example. Consider the diagram of $K$ whose arcs and crossings are labeled as in Figure~\ref{fig:4_1fig}. We may assume that all entries of $\boldsymbol{\epsilon}$ are $+1$, except possibly the one appearing in the final equation. More precisely, as we solve the equations successively, we set each $\epsilon_i$ appearing before the final equation equal to $+1$.

	\begin{figure}[!ht]
		\begin{center}
			\includegraphics[width=0.28\textwidth]{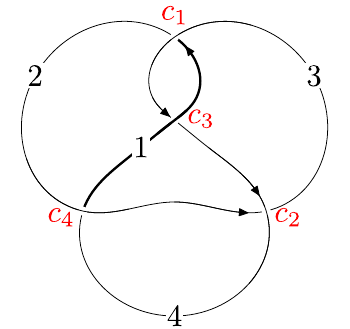}
		\end{center}
		\caption{A knot diagram of $4_1$.}
		\label{fig:4_1fig}
	\end{figure}
	
	We choose an ordering of the arcs as $1 \rightarrow 3 \rightarrow 2 \rightarrow 4$. Then, by Corollary~\ref{cor:normalizedrep}, we first consider the case where
	\begin{equation}
		\alpha_1=\binom{1}{0}, \ \alpha_3=\binom{0}{u}.
	\end{equation}
     At the crossings $c_1$ and $c_4$, we determine the other arc colors $\alpha_2$ and $\alpha_4$ as follows.
	\begin{align*}
		c_1 :& \quad \alpha_2=\alpha_1-\langle \alpha_1, \alpha_3 \rangle \alpha_3=\binom{1}{-u^2}, \\
		c_4 :& \quad \alpha_4=\alpha_1-\langle \alpha_1, \alpha_2 \rangle \alpha_2=\binom{1+u^2}{-u^4},    
	\end{align*}
	At the  crossing $c_2$, we have 
	$$c_2 : \quad \alpha_3 = \alpha_2 + \langle \alpha_2, \alpha_4 \rangle  \alpha_4$$
	which forces $u$ to satisfy $u^2+u+1=0$. The final equation, at the crossing $c_3$, is
    $$   c_3:\quad \pm\alpha_4 = \alpha_3 - \langle\alpha_3,\alpha_1\rangle\alpha_1.$$
    This equation holds for the minus sign, but not for the plus sign. We therefore obtain two nontrivial parabolic characters with negative obstruction class.
	
	The next case to consider, by Corollary~\ref{cor:normalizedrep}, is
	\begin{equation} \label{eqn.case}
		\alpha_1=\binom{1}{0}, \ \alpha_3=\binom{\ast}{0}, \ \alpha_2=\binom{0}{\ast}\,.
	\end{equation}
	However, a similar computation shows that, since $\alpha_3$ is proportional to $\alpha_1$, no nontrivial arc-coloring exists, and hence no nontrivial parabolic character is obtained. This concludes that
    \begin{equation}
		\calX^\times_+ (4_1)= \emptyset
		\quad \textrm{and} \quad
		\calX^\times_-(4_1) \simeq \Spec  \left(\BC[u^{\pm1}]/(u^2+u+1)\right).
	\end{equation}
    The two parabolic characters lie in the open chart $\langle \alpha_1,\alpha_3\rangle\neq 0$, and have multiplicity $1$.
    
	Applying the obtained arc-colorings to the formulas for cusp shapes and complex volumes in \cite{kim_octahedral_2018,kim_octahedral_2019}, we obtain Table~\ref{table:fig8_complex_volume}.
	
	\begin{table}[!h]
		\centering
		\renewcommand{\arraystretch}{1.2}
		\begin{tabular}{|r|r|r|}
			\hline
			 \text{Parameter} & \text{Complex volume} & \text{Cusp shape} \\
			\hline
			 $u =-0.5-0.86603\II$ & $2.02988\II$ & $-3.46410\II$ \\
			 $u = -0.5+0.86603\II$ & $-2.02988\II$ & $3.46410\II$ \\
			\hline
		\end{tabular}
		\caption{Complex volumes and cusp shapes of parabolic representations of the figure-eight knot.}
		\label{table:fig8_complex_volume}
	\end{table}

	\subsection{The knot $8_{18}$}\label{sec:8_18}
	Let us consider the knot $K=8_{18}$ with its diagram as in Figure~\ref{fig:8_18fig}. As explained, we may assume that every $\epsilon_i$ appearing before the final equation is $+1$.
	
	\begin{figure}[!h]
		\begin{center}
			\includegraphics[width=0.33\textwidth]{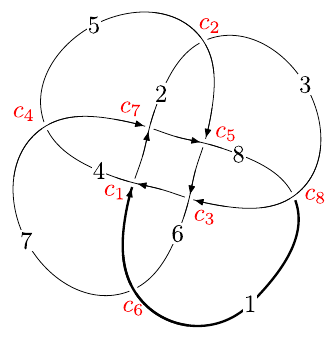}
		\end{center}
		\caption{A knot diagram of $8_{18}$.}\label{fig:8_18fig}
	\end{figure}
	
	We choose an ordering of the arcs as $1 \rightarrow 4 \rightarrow 6 \rightarrow \cdots$. By  Cororally~\ref{cor:normalizedrep}, the first case to consider is
	\begin{equation} \label{eqn.case2}
		\alpha_1=\binom{1}{0},\ \alpha_4 =\binom{0}{u}\,.
	\end{equation}
	Unlike in the previous section, $\alpha_1$ and $\alpha_4$ do not determine all the arc colors. For this reason, we introduce additional variables and set $\alpha_6=\binom{x}{y}$. Then, at the crossings $c_1, c_6, c_7, c_4, c_2$, we determine all the arc colors as follows.
	\begin{align*}
		c_1 : &\quad \alpha_2=\alpha_1-\lrbar{ \alpha_1, \alpha_4 } \alpha_4=\binom{ 1 }{ -u^2 },\\ 
		c_6 : & \quad \alpha_7=\alpha_6+\lrbar {\alpha_6, \alpha_1} \alpha_1=\binom{ x-y }{ y },\\ 
		c_7:& \quad \alpha_8=\alpha_7-\lrbar {\alpha_7, \alpha_2} \alpha_2=\binom{x + x u^2 - y u^2 }{ y - y u^2 - x u^4 + y u^4 },\\
		c_4 :& \quad \alpha_5=\alpha_4+\lrbar {\alpha_4, \alpha_7} \alpha_7=\binom{ -x^2 u + 2 x y u - y^2 u }{ u - x y u + y^2 u },\\
		c_2:&\quad \alpha_3=\alpha_2+\lrbar {\alpha_2, \alpha_5} \alpha_5=\begin{pmatrix*}
			\begin{aligned}
				1 - x^2 u^2 + 2 x y u^2 + x^3 y u^2 - y^2 u^2 - 3 x^2 y^2 u^2 + 
				3 x y^3 u^2 \\- y^4 u^2 + x^4 u^4 - 4 x^3 y u^4 + 6 x^2 y^2 u^4 - 
				4 x y^3 u^4 + y^4 u^4
			\end{aligned} \\[2.6ex]
			\begin{aligned}
				-2 x y u^2 + 2 y^2 u^2 + x^2 y^2 u^2 - 
				2 x y^3 u^2 + y^4 u^2 - x^2 u^4\\ + 2 x y u^4 + x^3 y u^4 - y^2 u^4 - 
				3 x^2 y^2 u^4 + 3 x y^3 u^4 - y^4 u^4
			\end{aligned}
		\end{pmatrix*}.
	\end{align*}
    At the crossings $c_3$ and $c_5$, we have
    \begin{align*}
		c_3 :& \quad \alpha_4 = \alpha_3 - \langle \alpha_3, \alpha_6\rangle \alpha_6, \\
		c_5 :& \quad \alpha_6 = \alpha_5 - \langle \alpha_5, \alpha_8\rangle \alpha_8 \, .
	\end{align*}
    These equations yield four polynomials $h_1,h_2,h_3,h_4\in\mathbb{Q}[u,x,y]$ that the variables $u,x,y$ should satisfy. Whenever $h_1=\cdots=h_4=0$, the last equation at $c_8$ 
    \begin{align*}
		c_8 : & \quad \pm \alpha_1 =\alpha_8+ \langle \alpha_8, \alpha_3 \rangle \alpha_3\,.
	\end{align*}
    is satisfied up to sign, where the sign is the obstruction class.
    
	Applying the primary decomposition algorithm to the ideal generated by $h_1$, $h_2$, $h_3$, $h_4$, the ideal decomposes into 9 components with the simpler generating sets listed in Table~\ref{table:8_18}. For instance, the first component $\calX^1$ of $\calX^\times(8_{18})$ is given by
    $$ \calX^1 \simeq \Spec \left( \BC[u^{\pm1},x,y]/(1+u+u^2,\,1+y+u,\,1+x) \right)$$
	and contains two $\BC$-points, hence two nontrivial parabolic characters with negative obstruction class.
    
    By Corollary~\ref{cor:normalizedrep}, the next case is
	\begin{equation} \label{eqn.case3}
		\alpha_1=\binom{1}{0}, \
		\alpha_4=\binom{u}{0}, \
		\alpha_6=\binom{0}{y}.
	\end{equation}
	A similar computation yields a single solution $(u,y)=(1,1)$ with negative obstruction class. In the remaining cases, both $\alpha_4$ and $\alpha_6$ are proportional to $\alpha_1$, which forces all arc colors to be proportional to $\alpha_1$. Hence, there are no further nontrivial arc-colorings.

    In conclusion, as listed in Table~\ref{table:8_18}, we have
    $$\calX^\times_{+}(8_{18})(\BC) = \{ \text{2 points} \} ,\qquad \calX^\times_{-}(8_{18})(\BC)=\{ \text{24 points} \} \,.$$ 
    The last point obtained in Case~\eqref{eqn.case3} lies in the open chart $\langle \alpha_1,\alpha_6\rangle\neq 0$, while all the other points lie in the open chart $\langle \alpha_1,\alpha_4\rangle\neq 0$. Computing their multiplicities on these charts, we find that every point has multiplicity $1$.

	\begin{table}[!h]
		\centering
		\renewcommand{\arraystretch}{1.3}
		\begin{tabular}{|c|c|c|p{0.52\textwidth}|c|}
			\hline
			\text{Case} & \text{Component} & \text{Obs.} & \text{Generators} &  \text{$\#$ points} \\
			\hline
			\eqref{eqn.case2}
			& $\calX^1$ & $+$ 
			& $1+u+u^2,\ 1+y+u,\ 1+x$
			& $2$ \\
			\cline{2-5}
			
			& $\calX^2$ & $-$ 
			& $-1+u,\ -1+y,\ x $
			& $1$ \\
			\cline{2-5}
			
			& $\calX^3$ & $-$ 
			& $1+u,\ y,\ 1+x$
			& $1$ \\
			\cline{2-5}
			
			& $\calX^4$ & $-$ 
			& $1+u,\ 1+y,\ 1+x$
			& $1$ \\
			\cline{2-5}
			
			& $\calX^5_{\mathrm{geom}}$ & $-$
			& $
			\begin{aligned}
				& -1+2u+u^2-2u^3+u^4,\\[-0.8ex]
				& -1+y+u^2-u^3,\ 1+x
			\end{aligned}
			$
			& $4$ \\
			\cline{2-5}
			
			& $\calX^6$ & $-$ 
			& $
			\begin{aligned}
				& 1+u+u^2, \\[-0.8ex]
				& -1-y+y^2-u-yu,\ x-u+yu
			\end{aligned}
			$
			& $4$ \\
			\cline{2-5}
			
			& $\calX^7$ & $-$ 
			& $
			\begin{aligned}
				& 3-6u+6u^2-3u^3+u^4, \\[-0.8ex]
				& -4+y+5u-3u^2+u^3,\ -3+3x-u^3
			\end{aligned}
			$
			& $4$ \\
			\cline{2-5}
			
			& $\calX^8$ & $-$ 
			& $
			\begin{aligned}
				& 1-2u+u^3+u^4, \\[-0.8ex]
				& -2+y+u+2u^2+u^3,\ x-2u-2u^2-u^3
			\end{aligned}$
			& $4$ \\
			\cline{2-5}
			
			& $\calX^{9}$ & $-$ 
			& $
			\begin{aligned}
				& 1-2u+u^3+u^4, \\[-0.8ex]
				& -2+y+2u+3u^2+2u^3,\ 1+x
			\end{aligned}
			$
			& $4$ \\
			\hline
			\eqref{eqn.case3}
			& $\calX^{10}$ & $-$ 
			& $-1+u,\ -1+y$
			& $1$ \\
			\hline
		\end{tabular}
		\caption{The $\mathbb{Q}$-components of $\calX^\times(8_{18})$.}
		\label{table:8_18}
	\end{table}

	Applying the complex volume formula to these $26$ parabolic characters, we obtain:
	\[
	\begin{gathered}
		\pm 1.71901,\quad
		\pm 12.3509\II,\quad
		1.64493 \pm 4.05977\II,\quad
		-1.64493 \pm 4.05977\II,\\
		\pm 4.05977\II,\quad
		\pm 1.64493\,.
	\end{gathered}
	\]
	Here the list records only distinct values; some of these values occur for more than one parabolic character. The maximum volume, that is, the largest imaginary part among these values, is $12.3509$. This agrees with the hyperbolic volume of the knot $8_{18}$ and is attained on the fifth component; see Table~\ref{table:8_18vol}.
	
	
	\begin{table}[!h]
		\centering
		\renewcommand{\arraystretch}{1.2}
		\begin{tabular}{|l|c|r|r|}
			\hline
			  \text{Parameter} & \text{Obs.} & \text{Complex volume} & \text{Cusp shape} \\
			\hline
			  $u=-0.88320$ & $-$ & $1.71901$ & $5.40877$ \\
			  $u=0.46899$ &  & $-1.71901$ & $-5.40877$ \\
			  $u=1.20711-0.97831\II$ &  & $-12.3509\II$ & $7.82655\II$ \\
			  $u=1.20711+0.97831\II$ &  & $12.3509\II$ & $-7.82655\II$ \\
			\hline
		\end{tabular}
		\caption{Complex volumes and cusp shapes for $\calX^5$.}
		\label{table:8_18vol}
	\end{table}

	\subsection{The knot $11n_{104}$}\label{sec:11n_104}
	
	Let us consider the knot $K=11n_{104}$ with the diagram shown in Figure~\ref{fig:11n_104fig}. It is one of simple examples for which a parabolic character occurs with multiplicity greater than one.

	\begin{figure}[!h]
		\begin{center}
			\includegraphics[width=0.35\textwidth]{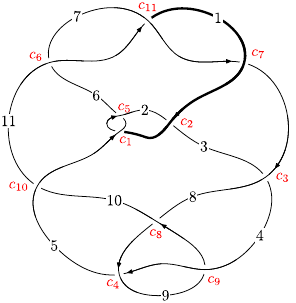}
		\end{center}
		\caption{A knot diagram of $11n_{104}$.}\label{fig:11n_104fig}
	\end{figure}
	
	We choose an ordering of the arcs as $1 \rightarrow 5 \rightarrow 7 \rightarrow \cdots$.  By Cororally~\ref{cor:normalizedrep}, the first case to consider is
	$$
	\alpha_1=\binom{1}{0}, \quad \alpha_5=\binom{0}{u}.
	$$
	As in the previous section, we set $\alpha_7=\binom{x}{y}$ and determine ther other arc colors at the crossings $c_1$, $c_2$, $c_7$, $c_{11}$, $c_6$, $c_{10}$, $c_8$ and $c_4$.
	We omit the explicit expressions for the colors, as some of them are too lengthy. They are recorded in the Mathematica notebook accompanying the arXiv version of this paper and are also available at~\cite{diagramsite}.

	At the crossings $c_3$ and $c_5$,
	$$
	\begin{aligned}
		c_3 :& \quad \alpha_4 = \alpha_3 + \langle \alpha_3, \alpha_8\rangle \alpha_8, \\
		c_5 :& \quad \alpha_6 = \alpha_5 + \langle \alpha_5, \alpha_2\rangle \alpha_2 \, ,
	\end{aligned}
	$$
	we obtain four polynomials $h_1,h_2,h_3,h_4 \in \mathbb{Q}[u,x,y]$. Applying the primary decomposition algorithm to the ideal generated by $h_1,\ldots,h_4$, the ideal decomposes into three components with the simpler generating sets listed in Table~\ref{table:11n_104}. Whenever $h_1=\cdots=h_4=0$, the final equation at $c_9$,
	$$
	c_9:\quad \pm \alpha_{10}=\alpha_9+\langle \alpha_9,\alpha_4\rangle \alpha_4,
	$$
    holds up to sign, and the sign in the left-hand side determines the obstruction class. The plus sign occurs on the components $\calX^1$ and $\calX^2$, while the minus sign occurs on $\calX^3$.
	
	In the remaining cases, $\alpha_5$ is proportional to $\alpha_1$, and this  forces all arc colors to be proportional to $\alpha_1$. Hence, there are no further nontrivial arc-colorings.
	
	\begin{table}[!h]
		\centering
		\renewcommand{\arraystretch}{1.3}
		\small
		\begin{tabularx}{\textwidth}{|c|c|c|>{\raggedright\arraybackslash}X|c|}
			\hline
			\text{Comp.} & \text{Obs.} & \text{Mult.}
			& \text{Generators} & \text{$\#$ points} \\
			\hline
			
			$\calX^1$ & $+$ & $3$
			& $1+y,\quad (1+x)^3,\quad -1+u$
			& $1$ \\
			\hline
			
			$\calX^2$ & $+$ & $1$
			& $-1+y,\quad 1+4x^2-4x^3+x^4,\quad 1+u$
			& $4$ \\
			\hline
			
			$\calX^3_{\mathrm{geom}}$ & $-$ & $1$
			&
			$\begin{aligned}
				&1+8x+14u-u^2+28u^3-9u^4-2u^5+u^6,\qquad y-u,\\
				&-1+2u+2u^2-14u^3+8u^4 +30u^5-10u^6-2u^7+u^8
			\end{aligned}
			$
			& $8$ \\
			\hline
		\end{tabularx}
		\caption{The $\mathbb{Q}$-components of $\calX^\times(11n_{104})$.}
		\label{table:11n_104}
	\end{table}
	
	The first component $\calX^1$  consists of a single point of multiplicity three. Thus, counted with multiplicity, $\calX^\times_+(11n_{104})$ and $\calX^\times_-(11n_{104})$ consist of $7$ and $8$ $\BC$-points, respectively. This agrees with the conjecture of B\'{e}nard and Detcherry:
	$$
	8-7=\frac{\det(11n_{104})-1}{2} = \frac{3-1}{2}.
	$$
	Obviously, the equation fails if we ignore the multiplicity.

	\section{Computations for knots with at most 12 crossings}\label{sec.bd}
	

    Using the quandle method described in Section~\ref{sec.example}, together with computer calculations, the third author and Philip Choi computed complete lists of nontrivial parabolic characters for all $2977$ knots with at most $12$ crossings, including their multiplicities and obstruction classes. The resulting data are available at~\cite{diagramsite}. We summarize the results below.

	Of the $2977$ knots with at most $12$ crossings, exactly $93$ have a positive-dimensional component in their parabolic character schemes. These are
	{\small
		\setlength{\tabcolsep}{4pt}
		\begin{longtable}{*{8}{l}}
			$10_{98}$ & $10_{99}$ & $10_{123}$ & $11a_{43}$ & $11a_{44}$ & $11a_{47}$ & $11a_{57}$ & $11a_{132}$ \\
			$11a_{157}$ & $11a_{231}$ & $11a_{263}$ & $11n_{71}$ & $11n_{72}$ & $11n_{73}$ & $11n_{74}$ & $11n_{75}$ \\
			$11n_{76}$ & $11n_{77}$ & $11n_{78}$ & $11n_{81}$ & $12a_{29}$ & $12a_{30}$ & $12a_{33}$ & $12a_{36}$ \\
			$12a_{113}$ & $12a_{114}$ & $12a_{116}$ & $12a_{117}$ & $12a_{119}$ & $12a_{122}$ & $12a_{157}$ & $12a_{164}$ \\
			$12a_{166}$ & $12a_{167}$ & $12a_{182}$ & $12a_{195}$ & $12a_{348}$ & $12a_{427}$ & $12a_{435}$ & $12a_{554}$ \\
			$12a_{623}$ & $12a_{647}$ & $12a_{668}$ & $12a_{692}$ & $12a_{693}$ & $12a_{694}$ & $12a_{701}$ & $12a_{750}$ \\
			$12a_{801}$ & $12a_{923}$ & $12a_{924}$ & $12a_{981}$ & $12a_{982}$ & $12a_{987}$ & $12a_{990}$ & $12a_{1019}$ \\
			$12a_{1105}$ & $12a_{1202}$ & $12a_{1225}$ & $12a_{1288}$ & $12n_{55}$ & $12n_{56}$ & $12n_{57}$ & $12n_{58}$ \\
			$12n_{59}$ & $12n_{60}$ & $12n_{61}$ & $12n_{62}$ & $12n_{63}$ & $12n_{64}$ & $12n_{66}$ & $12n_{67}$ \\
			$12n_{219}$ & $12n_{220}$ & $12n_{221}$ & $12n_{222}$ & $12n_{223}$ & $12n_{224}$ & $12n_{225}$ & $12n_{229}$ \\
			$12n_{261}$ & $12n_{440}$ & $12n_{508}$ & $12n_{518}$ & $12n_{553}$ & $12n_{554}$ & $12n_{555}$ & $12n_{556}$ \\
			$12n_{604}$ & $12n_{605}$ & $12n_{642}$ & $12n_{706}$ & $12n_{888}$ & & &
		\end{longtable}
	}
	\noindent
	For the remaining $2884$ knots, the parabolic character scheme is finite. We verified that
	\begin{itemize}
		\item all $1958$ small knots satisfy the B\'{e}nard--Detcherry conjecture; and
		\item of the $926$ large knots with finite parabolic character variety, all but ten satisfy the conjecture.
	\end{itemize}
	Thus Equation~\eqref{conj.bd} holds for $2874$ of the $2884$ knots with finite parabolic character scheme, the ten exceptions all being large. The exceptions are
	\[
	\begin{array}{*{5}{l}}
		12a_{396} & 12a_{634} & 12a_{1124} & 12a_{1167} & 12n_{494} \\
		12n_{495} & 12n_{496} & 12n_{600} & 12n_{601} & 12n_{602}
	\end{array}
	\]
	and have finite and reduced parabolic character schemes. It is worth noting that the discrepancy is $\pm2$:
	\begin{equation*}
	n_- - n_+ = \frac{\det(K)-1}{2}\pm 2\,,
	\end{equation*}
	with the minus sign for $12a_{1124}$ and the plus sign for the other nine knots. 
    \begin{remark}
    According to \cite{knotinfo}, the Alexander module of each of these ten knots is non-cyclic over $\mathbb{Z}[t^{\pm1}]$ (with Nakanishi index 2) but becomes cyclic over $\mathbb{Q}[t^{\pm1}]$. By contrast, all $29$ knots in our list whose Alexander modules remain non-cyclic over $\mathbb{Q}[t^{\pm1}]$ satisfy the conjecture.     
    \end{remark}
    
	Among the 2884 knots with finite parabolic character scheme, exactly $43$ have a parabolic character of multiplicity greater than one. In each case, the relevant elimination ideal contains an irreducible univariate factor of degree one or two, occurring with multiplicity two or three. The possible types are as follows.
    \begin{itemize}
        \item A linear component of multiplicity three with positive obstruction ($24$ knots):
	{\small
		\[
		\begin{array}{*{8}{l}}
			11a_{100} & 11a_{108} & 11a_{109} & 11a_{134} & 11a_{194} & 11a_{223} & 11n_{85} & 11n_{87} \\
			11n_{104} & 11n_{105} & 11n_{106} & 11n_{107} & 12a_{280} & 12a_{281} & 12a_{374} & 12a_{445} \\
			12a_{702} & 12a_{706} & 12a_{708} & 12a_{737} & 12a_{739} & 12a_{741} & 12a_{771} & 12a_{798}
		\end{array}
		\]
	}%
    \item A quadratic component of multiplicity three with positive obstruction ($16$ knots):
	{\small
		\[
		\begin{array}{*{8}{l}}
			12a_{217} & 12a_{220} & 12a_{225} & 12a_{250} & 12a_{263} & 12a_{357} & 12a_{382} & 12a_{383} \\
			12n_{142} & 12n_{143} & 12n_{271} & 12n_{272} & 12n_{295} & 12n_{296} & 12n_{300} & 12n_{301}
		\end{array}
		\]
	}%
    \item A quadratic component of multiplicity two with positive obstruction ($2$ knots): $12a_{869}$ and $12a_{1152}$.
    \item A quadratic component of multiplicity two with negative obstruction ($1$ knot): $12n_{881}$
    \end{itemize}

	\subsection*{Acknowledgments}
	The authors are grateful to Dong-il Lee, Dong Uk Lee, Philip Choi and Renaud Detcherry for helpful discussions. 
	The knot diagrams were drawn with the \texttt{draw programme} developed by Andrew Bartholomew (\url{http://www.layer8.co.uk/maths/draw/index.htm}), modified in part for our purposes.
	S.Y. was supported by the National Research Foundation of Korea (NRF) grant funded by the Ministry of Education (RS-2024-00442775).
	S.K. was supported by Basic Science Research Program through the National Research Foundation of Korea (NRF) funded by the Ministry of Education (RS-2025-25434700).

	\bibliographystyle{amsalpha}
	\bibliography{bibliog}
\end{document}